\documentclass[10pt,final,twocolumn]{ieeeconf} 
\IEEEoverridecommandlockouts

\usepackage{epsfig}
\usepackage{amsmath}
\usepackage{amssymb}
\usepackage{indentfirst}
\usepackage{cite}
\usepackage{bm}
\usepackage{graphicx}           
\usepackage{psfrag}             
\usepackage{subfigure}          
\usepackage{empheq}
\usepackage{enumerate}

\usepackage{enumerate}
\usepackage{empheq}
\usepackage{color}
\usepackage{tikz}
\usepackage{empheq}
\usetikzlibrary{arrows}
\usetikzlibrary{positioning}

\newcommand{\R}{\mathbb{R}}

\DeclareMathOperator{\rank}{rank}
\DeclareMathOperator{\Vect}{vec}
\DeclareMathOperator{\sgn}{sgn}

\newcommand{\s}{\bm{s}}
\newcommand{\p}{\bm{p}}

\newcommand{\hb}{\mathbf{h}}

\newcommand{\vb}{\mathbf{v}}
\newcommand{\wb}{\mathbf{w}}
\newcommand{\xb}{\mathbf{x}}

\newcommand{\zb}{\mathbf{z}}
\newcommand{\qb}{\mathbf{q}}

\newcommand{\phib}{\bm{\phi}}

\newcommand{\Hb}{\mathbf{H}}
\newcommand{\Pb}{\mathbf{P}}

\newcommand{\AC}{\mathcal{I}}

\usepackage{color}
\definecolor{gray}{RGB}{128,128,128}

\usepackage{tikz}
\usetikzlibrary{arrows}

\newenvironment{proof}[1][Proof]%
  {\smallskip\par\noindent\textbf{#1\,:\ }}%
  {\hspace*{\fill} \rule{6pt}{6pt}\smallskip}
\newenvironment{proof*}[1][Proof]%
  {\medskip\par\noindent\textbf{#1\,:\ }}%

\newtheorem{remark}{\textbf{Remark}}

\newtheorem{assumption}{\emph{\textbf{Assumption}}}
\newtheorem{theorem}{\textbf{Theorem}}
\newtheorem{lemma}{\textbf{Lemma}}

\newtheorem{proposition}{\textbf{Proposition}}

\usepackage{algorithm}
\usepackage{algorithmic}
\usepackage{setspace}
\let\Algorithm\algorithm
\renewcommand\algorithm[1][]{\Algorithm[#1]\setstretch{1.4}}

\allowdisplaybreaks[4]

\title{A Finite-Time Algorithm for the Distributed Tracking of Maneuvering Target}
\author{Jemin~George\\
CCDC Army Research Laboratory\\ Adelphi, MD 20783, USA}
\allowdisplaybreaks[4]

\begin{document}


%
\maketitle
\begin{abstract}
This paper presents a novel distributed algorithm for tracking a maneuvering target using bearing or direction of arrival measurements collected by a networked sensor array. The proposed approach is built on the dynamic average-consensus algorithm, which allows a networked group of agents (nodes) to reach consensus on the global average of a set of local time-varying signals in a distributed fashion. Since the average-consensus error corresponding to the presented dynamic average-consensus algorithm converges to zero in finite time, the proposed distributed algorithm guarantees that the tracking error converges to zero in finite time. Numerical simulations are provided to illustrate the effectiveness of the proposed algorithm.
\end{abstract}

\section{Introduction}

We study the problem of distributedly tracking a maneuvering target using direction of arrival or bearing measurements. This problem arises in numerous surveillance and reconnaissance applications where a set of networked sensors are tasked with jointly tracking a maneuvering target without the aid of a centralized fusion node to pool all the local observations.  Tracking a maneuvering target is a formidable problem because it is impossible to come up with a single motion model that can account for all possible target maneuvers.  Thus the current solution to tracking a maneuvering target involves multiple-model methods, where a bank of motion models are used to approximate the target motion (e.g., Multiple Model Adaptive Estimator (MMAE)~\cite{1982stochastic, Maybeck1985, Maybeck95}, Interacting Multiple Model (IMM) estimator~\cite{Blom1984, Mazor98IMM, Blom1998, Kirubarajan2003}, Variable Structure Multiple Model (VSMM) estimator~\cite{LiCDC1994, Li96VSMM, LiTAC2000}). Though there exist several distributed implementations of multiple model estimators~\cite{Hong2004TAC, Athalye2006, Li2012TAC, Wang2016CDC, Wang2016Access, Yu2016, Yu2017}, they are unable to precisely recover the performance of the centralized algorithm due to the inability of distributed consensus methods to instantaneously reach agreement on the mode-dependent target dynamics and the innovations process. Moreover, even the centralized multiple model estimators are not guaranteed to precisely track a highly maneuvering target.

This paper presents a novel distributed algorithm that allows the networked agents (sensors) to precisely track a highly maneuvering target from bearing measurements. More precisely, the proposed algorithm guarantees that the tracking error converges to zero in finite time. The proposed approach is built on recent advances in \emph{dynamic average-consensus} algorithms~\cite{WeiRen2012TAC, Kia2014, George2017ACC, George2017CDC, George2018arXiv, george2018distributed, george2018distributed, yang2018distributed} that allow individual nodes to estimate the global average of the local time-varying signals of interest. Though there exist numerous dynamic average consensus algorithms, the approach proposed here is robust to network dynamics~\cite{du2018distributed, george2019robust, george2019distributed, burbano2019inferring, george2019distributed, george2019fast, george2019distributedDL, zhang2019computational, yang2020distributed, du2019event}. Compared to the centralized approach, the proposed distributed scheme is more robust to network disruptions, avoids a single point of failure, and is easily scalable with the number of agents in the network.

The rest of this paper is organized as follows. Mathematical preliminaries and the detailed problem formulation are given in Sections \ref{sec:pre} and \ref{sec:prob}, respectively. Main results of the paper are given in Section \ref{sec:Main}. Section \ref{sec:sim} provides the results obtained from numerical simulations. Conclusions and future work are discussed in Section~\ref{sec:con}. 

\section{Preliminaries}\label{sec:pre}
\vspace{-2mm}

Let $\mathbb{R}^{n\times m}$ denote the set of $n\times m$ real matrices. An $n\times n$ identity matrix is denoted as $I_n$ and $\mathbf{1}_n$ denotes an $n$-dimensional vector of all ones. For two vectors $\mathbf{x} \in \mathbb{R}^n$ and $\mathbf{y} \in \mathbb{R}^n$, $\mathbf{x} \geq  \mathbf{y} \,\,\left(\mathbf{x} \leq \mathbf{y} \right)$ implies $x_i \geq  y_i, \,\, \left({x}_i \leq {y}_i \right)$, $\forall\,i\in\{1,\ldots,n\}$. The absolute value of a vector is given as $|\mathbf{x}| = \begin{bmatrix} |x_1| & \ldots & |x_n|\end{bmatrix}^T$. Let $\text{sgn}\{\cdot\}$ denote the signum function, 
and $\forall\,\mathbf{x}\in\mathbb{R}^n$, $\text{sgn}\{\mathbf{x}\} \triangleq \begin{bmatrix} \text{sgn}\{x_1\} & \ldots & \text{sgn}\{x_n\}\end{bmatrix}^T$. For $p\in[1,\,\infty]$, the $p$-norm of a vector $\mathbf{x}$ is denoted as $\left\| \mathbf{x} \right\|_p$. For matrices $A \in \mathbb{R}^{m\times n}$ and $B \in \mathbb{R}^{p \times q}$, $A \otimes B \in \mathbb{R}^{mp \times nq}$ denotes their Kronecker product.


For an \emph{undirected} graph $\mathcal{G}\left(\mathcal{V},\mathcal{E}\right)$ of order $n$, $\mathcal{V} \triangleq \left\{v_1, \ldots, v_n\right\}$ represents the sensors or nodes. The communication links between the sensors are represented as $\mathcal{E} \triangleq \left\{e_1, \ldots, e_{\ell}\right\} \subseteq \mathcal{V} \times \mathcal{V}$. Here each undirected edge is considered as two distinct directed edges and the edges are labeled such that they are grouped into incoming links to nodes $v_1$ to $v_n$. Let $\mathcal{I}$ denote the index set $\{1,\ldots,n\}$ and $\forall i \in \mathcal{I}$; let $\mathcal{N}_i \triangleq \left\{v_j \in \mathcal{V}~:~(v_i,v_j)\in\mathcal{E}\right\}$ denote the set of neighbors of node $v_i$. Let $\mathcal{A} \triangleq \left[a_{ij}\right]\in \{0,1\}^{n\times n}$ be the \emph{adjacency matrix} with entries $a_{ij} = 1 $ if $(v_i,v_j)\in\mathcal{E}$ and zero otherwise. Define $\Delta \triangleq \text{diag}\left(\mathcal{A}\mathbf{1}_n\right)$ as the degree matrix associated with the graph and $\mathcal{L} \triangleq \Delta - \mathcal{A}$ as the graph \emph{Laplacian}. The \emph{incidence matrix} of the graph is defined as $\mathcal{B} = \left[b_{ij}\right]\in \left\{-1,0,1\right\}^{n\times \ell}$, where $b_{ij} = -1$ if edge $e_j$ leaves node $v_i$, $b_{ij} = 1$ if edge $e_j$ enters node $v_i$, and $b_{ij} = 0$ otherwise. 

\section{Problem Formulation}\label{sec:prob}
\vspace{-2mm}
Consider the problem of tracking a maneuvering target using a stationary sensor network of $n$ sensors located at positions, $\s_{i}\in\R^2$, $i\in\AC$. The sensor positions are locally known to each node. We model the sensor network as an undirected graph $\mathcal{G}\left(\mathcal{V},\mathcal{E}\right)$ of order $n$, where the nodes represent the sensors and the edges denote the communication links between them. Sensors $v_i$ and $v_j$ are (one-hop) neighbors if $(v_i,v_j)\in\mathcal{E}$. We assume that all sensor are synchronized with a common clock and each sensor can only communicate with its neighboring sensors. Sensors obtain bearing observations to a maneuvering target at position $\p(t)\in\mathbb{R}^2$, whose kinematics is given as
\begin{align}\label{target}
  \dot{\p}(t) &= \vb(t),
\end{align}
where $\vb(t)\in\mathbb{R}^2$ is the target velocity. The bearing measurements are represented as unit vectors $\bm{\varphi}_i(t) \in \mathcal{S}^1$ of the form\footnote{Here $\mathcal{S}^1$ denotes the set of unit-norm vectors in $\mathbb{R}^2$.}
\begin{align}\label{observation}
  \bm{\varphi}_i(t) &=  
  \frac{ \p(t)-\s_i }{\| \p(t)-\s_i \|_2},\qquad i\in\AC.
\end{align}
Let $\theta_i(t) \in [0,2\pi)$ denote the bearing angle, measured positive counter-clockwise, measured by the $i$-th agent, and define $\rho_i(t) = {\| \p(t)-\s_i \|_2}$. Thus
$\bm{\varphi}_i(t) = \begin{bmatrix} \cos\left(\theta_i(t)\right) & \sin\left(\theta_i(t)\right)\end{bmatrix}^\top,$ and
\begin{equation}\label{rhoEqn1}
  \rho_i(t)\bm{\varphi}_i(t) = \p(t)-\s_i.
\end{equation}
Therefore, the distributed target tracking problem consists of each sensor estimating the target trajectory $\p(t)$ from its own bearing measurements and any information obtained from its neighbors as defined by the network topology.
\begin{proposition}\label{Prop1}
For any $i\in\AC$ and $\theta_i(t) \in [0,2\pi)$, let $\bm{\varphi}_i(t) = \begin{bmatrix} \cos\left(\theta_i(t)\right) & \sin\left(\theta_i(t)\right)\end{bmatrix}^\top$ and let $\bar{\bm{\varphi}}_i(t) \in \mathcal{S}^1$ be an orthonormal vector obtained by rotating $\bm{\varphi}_i(t)$ by $\pi/2$ radians clockwise. Then
\begin{enumerate}[(i)~]
    \setlength\itemsep{1em}
  \item $\bar{\bm{\varphi}}_i(t) = \begin{bmatrix} -\sin\left(\theta_i(t)\right) & \cos\left(\theta_i(t)\right)\end{bmatrix}^\top$, and
  \item $\bm{\varphi}_i(t) \bm{\varphi}^\top_i(t) + \bar{\bm{\varphi}}_i(t)\bar{\bm{\varphi}}^\top_i(t) = I_2$.
\end{enumerate}
\end{proposition}
\begin{proof}
  The proof follows from noticing that $$\bar{\bm{\varphi}}_i(t) = \begin{bmatrix} \cos(-\pi/2) & \sin(-\pi/2) \\ -\sin(-\pi/2) & \cos(-\pi/2)\end{bmatrix} \bm{\varphi}_i(t)$$ and $\cos^2(\theta_i) + \sin^2(\theta_i) = 1$ for all $\theta_i \in [0,2\pi) $.
\end{proof}\\
Multiplying \eqref{rhoEqn1} with $\bar{\bm{\varphi}}^\top_i(t)$ yields
\begin{equation}\label{rhoEqn2}
  \bar{\bm{\varphi}}^\top_i(t)\s_i = \bar{\bm{\varphi}}^\top_i(t)\p(t).
\end{equation}
Note that $\bar{\bm{\varphi}}^\top_i(t)\s_i$ is a scalar known to each agent. Now define
\begin{equation}\label{Heqn}
  \Hb(t) = \begin{bmatrix}
             \hb^\top_1(t) \\ \hb^\top_2(t) \\ \vdots \\ \hb^\top_n(t)
           \end{bmatrix},\quad
  \zb(t) = \begin{bmatrix}
             z_1(t) \\ z_2(t) \\ \vdots \\ z_n(t)
           \end{bmatrix}
\end{equation}
where $\hb^\top_i(t) = \bar{\bm{\varphi}}^\top_i(t)$ is the $i$-th row vector of matrix $\Hb(t)\in\R^{n\times 2}$ and $z_i(t) = \bar{\bm{\varphi}}^\top_i(t)\s_i$ is the $i$-th element of $\zb(t)\in\R^n$. Now \eqref{rhoEqn2} for the entire sensor network can be written as
\begin{equation}\label{LEQ11}
  \zb(t) = \Hb(t)\p(t).
\end{equation}
Thus, estimating the target trajectory corresponds to solving a linear time-varying set of equations. We make the following assumption regarding $\Hb(t)$:
\begin{assumption}\label{Assump1}
  For all $t\geq 0$, $\rank\left(\Hb(t)\right) = 2 < n$.
\end{assumption}
\noindent We aim to find the trajectory $\p(t)$ that minimizes or solves the following optimization problem:
\begin{equation}\label{Opt}
  \min_{\p(t)\,\in\,\R^2}\,\frac{1}{2} \| \zb(t) - \Hb(t)\p(t) \|^2.
\end{equation}
Under Assumption~\ref{Assump1}, the problem \eqref{Opt} has a unique solution:
\begin{equation}\label{yopt}
  \p^*(t) = \left( \Hb^\top(t)\Hb(t) \right)^{-1}\Hb^\top(t)\zb(t).
\end{equation}
In this paper, we aim to develop a distributed algorithm to solve the optimization problem \eqref{Opt} via local interactions dictated by the network topology.

\section{Proposed Distributed Algorithm}\label{sec:Main}
\vspace{-2mm}
In this section, we present a distributed algorithm for solving the least-squares problem \eqref{Opt}. In terms of local quantities, the least-squares solution in \eqref{yopt} can be written as
\begin{equation}\label{yopt1}
  \p^*(t) = \left( \frac{1}{n}\,\sum_{i=1}^{n} \hb_i(t)\hb^\top_i(t) \right)^{-1}\left( \frac{1}{n}\,\sum_{i=1}^{n} \hb_i(t) z_i(t) \right).
\end{equation}
Thus the optimal estimates can be obtained distributedly if the sensors can reach average consensus on $\bar{\Pb}(t) \in\R^{2\times 2}$ and $\bar{\qb}(t)\in\R^{2}$, where
\begin{align*}
  \bar{\Pb}(t) &=  \frac{1}{n}\,\sum_{i=1}^{n} \hb_i(t)\hb^\top_i(t),\quad\text{and}\quad  \bar{\qb}(t) = \frac{1}{n}\,\sum_{i=1}^{n} \hb_i(t) z_i(t).
\end{align*}
In summary, if sensors needs to reach consensus on the symmetric matrix $\bar{\Pb}(t)$ and the vector $\bar{\qb}(t)$, then the optimal solution can be computed as
\begin{equation}\label{yopt1a}
  \p^*(t) = \left( \bar{\Pb}(t) \right)^{-1} \bar{\qb}(t) .
\end{equation}

Here we propose a robust dynamic average-consensus algorithm to reach consensus on the time-varying quantities $\bar{\Pb}(t)$ and $\bar{\qb}(t)$. Toward this goal, we first define
\begin{align}
\Pb_i(t) &= \hb_i(t)\hb^\top_i(t),\label{PiEqn}\\
\qb_i(t) &= z_i(t)\hb_i(t). \label{qiEqn}
\end{align}
Note that quantities $\Pb_i(t)$ and $\qb_i(t)$ are locally available to the sensor. Thus,
$$\bar{\Pb}(t) =  \frac{1}{n}\,\sum\limits_{i=1}^{n} \Pb_i(t)\quad\text{and}\quad\bar{\qb}(t) = \frac{1}{n}\,\sum\limits_{i=1}^{n} \qb_i(t).$$
Construct a vector $\phib_i(t) \in \R^{6}$ containing the $4$ elements of the matrix $\Pb_i(t)$ and $2$ elements of $\qb_i(t)$, i.e.,
\begin{equation}\label{PhiEqn}
  \phib_i(t) = \begin{bmatrix}
               \Vect\left( \Pb_i(t) \right) \\ \qb_i(t)
             \end{bmatrix}.
\end{equation}
Before we proceed, we make the following assumption regarding $z_i(t)$ and $\hb_i(t)$:
\begin{assumption}\label{Assump:Phi}
Signals $z_i(t)$ and $\hb_i(t)$ are bounded and continuously differentiable with bounded derivatives such that there exists a positive constant $\gamma > 0$ such that $\forall\, i\in\AC$
\begin{align}
  &\sup_{\substack{t\in[t_0,\infty)}} \| \dot{\phib}_i(t) \|_{\infty} \leq \gamma < \infty. \label{Eq:Phi:Assump}
\end{align}
\end{assumption}
\begin{remark}
In the context of the tracking example discussed in the previous section, Assumption~\ref{Assump:Phi} corresponds to a known bound on target velocity.
\end{remark}
Here we propose a dynamic average-consensus algorithm that would allow each agent to estimate the time-varying signal
\begin{align}
\bar{{\phib}}(t) = \frac{1}{n} \sum_{i=1}^{n}\,{\phib}_i(t) = {\frac{1}{n} \left(\mathbf{1}^{\top}_n \otimes I_{6}\right) \bm{\phi}(t),}
\end{align}
where $\bm{\phi}(t) \in$ $\mathbb{R}^{n 6} \triangleq \begin{bmatrix} {\phib}^{\top}_1(t) & \ldots & {\phib}^{\top}_n(t)\end{bmatrix}^{\top}$. Assumption~\ref{Assump:Phi} ensures that the rate of change of $\bar{{\phib}}(t)$ is bounded such that the sensors are able to reach consensus on $\bar{{\phib}}(t)$.
Now we make following standing assumption regarding the network topology.
\begin{assumption}\label{Assump:Graph}
The interaction topology of $n$ networked sensors is given as an unweighted connected undirected graph $\mathcal{G}\left(\mathcal{V},\mathcal{E}\right)$.
\end{assumption}

\begin{lemma}\label{Lemma1}
For any strongly connected, weight-balanced graph $\mathcal{G}\left(\mathcal{V},\mathcal{E}\right)$ of order $n$, the graph Laplacian $\mathcal{L}$ is a positive semi-definite matrix with a single eigenvalue at $0$ corresponding to both the left and right eigenvectors $\mathbf{1}^\top_n$ and $\mathbf{1}_n$, respectively.
\end{lemma}
\begin{proof}
  See \cite{OlfatiSaber2004}.
\end{proof}

\begin{remark}
For all $\mathbf{x}\in\mathbb{R}^n$, such that $\mathbf{1}^{\top}_n\mathbf{x} = 0$, we have $\mathbf{x}^{\top}L\left( \mathcal{L}\mathcal{L} \right)^+\mathbf{x} =  \mathbf{x}^{\top}\mathbf{x} $ and $\mathbf{x}^{\top}\mathcal{L}\,\mathbf{x} \geq  \lambda_2(\mathcal{L})\|\mathbf{x}\|^2_2$. Here $\lambda_2(\mathcal{L})$ denotes the second-smallest eigenvalue of $\mathcal{L}$ or the algebraic connectivity of $\mathcal{G}\left(\mathcal{V},\mathcal{E}\right)$.
\end{remark}

\begin{lemma}\label{Lemma2}
Let $M \triangleq \left(I_n - \displaystyle\frac{1}{n}\mathbf{1}_n\mathbf{1}^{\top}_n\right)$. For any connected undirected network $\mathcal{G}\left(\mathcal{V},\mathcal{E}\right)$ of order $n$, the graph Laplacian $\mathcal{L}$ and the incidence matrix $\mathcal{B}$ satisfy
\begin{align}\label{Eq:M}
  M = \mathcal{L} \left( \mathcal{L }\right)^+ = \mathcal{BB}^{\top}\left(\mathcal{BB}^{\top}\right)^+ = \mathcal{B}\left(\mathcal{B}^{\top}\mathcal{B}\right)^+\mathcal{B}^{\top},
\end{align}
where $\left(\cdot\right)^+$ denotes the generalized inverse.
\end{lemma}
\begin{proof}
See Lemma~3~of~\cite{Gutman04}.
\end{proof}

\vspace{-3mm}
\subsection{Dynamic Average-Consensus Algorithm}
\vspace{-2mm}
Here we propose the following dynamic average-consensus algorithm:
\begin{subequations}\label{RDATsys1}
\begin{align}
    \begin{split}
        \dot{\bm{w}}_i(t) &= - \beta\sum\limits_{j=1}^{n} a_{ij} \sgn \bigg\{ \bm{x}_i(t)-\bm{x}_j(t) \bigg\},\quad \bm{w}_i(t_0),
    \end{split}\label{RDATsys1a}\\
  \bm{x}_i(t) &= \bm{w}_i(t) + \bm{\phi}_i(t)\label{RDATsys1b},
\end{align}
\end{subequations}
where $\bm{w}_i(t) \in \mathbb{R}^{6}$ is the internal states associated with the $i^{\text{th}}$ node, $\bm{x}_i(t) \in \mathbb{R}^{6}$  denotes the $i^{\text{th}}$ node's estimate of $\bar{\phib}(t)$, and $\beta > 0$ is a scalar parameter to be determined. The algorithm in \eqref{RDATsys1} can be rewritten in a compact form:
\begin{subequations}\label{RDATsys2}
\begin{align}
  \dot{\wb}(t) &=  -\beta\, \left( \mathcal{B} \otimes I_{6}\right) \sgn\left\{ \left( \mathcal{B}^{\top} \otimes I_{6}\right){\mathbf{x}}(t)\right\},\quad\wb(t_0)\label{RDATsys2a}\\
  \xb(t) &= \wb(t) + \bm{\phi}(t)\label{RDATsys2b},
\end{align}
\end{subequations}
where $\mathbf{x}(t)$ $\in$ $\mathbb{R}^{n 6}$ $\triangleq$ $\begin{bmatrix} \bm{x}^{\top}_1(t) & \ldots & \bm{x}^{\top}_n(t)\end{bmatrix}^{\top}$ is the estimate of $\bar{\phib}(t)$ for the entire network and $\mathbf{w}(t)$ $\in$ $\mathbb{R}^{n 6}$ $\triangleq$ $\begin{bmatrix} \bm{w}^{\top}_1(t) & \ldots & \bm{w}^{\top}_n(t)\end{bmatrix}^{\top}$ are the internal states of the algorithm for the entire network. Let $\tilde{\mathbf{x}}(t)\triangleq\mathbf{x}(t)-\mathbf{1}_n \otimes \bar{\phib}(t)$ denote the dynamic average-consensus error for the entire network. From \eqref{RDATsys2b} and Lemma~\ref{Lemma2}, we have
\begin{align}\label{AvgConErr}
    \tilde{\mathbf{x}}(t) = \wb(t) + \left( M \otimes I_{6}\right) \bm{\phi}(t).
\end{align}
Convergence analysis of the proposed algorithm is given next.

\begin{figure*}[!ht]
  \begin{centering}
      \subfigure[Simulation scenario]{
      \psfrag{X-location}[][]{\footnotesize{$x$-location}}
      \psfrag{Y-location}[][]{\footnotesize{$y$-location}}
      \includegraphics[width=.3\textwidth]{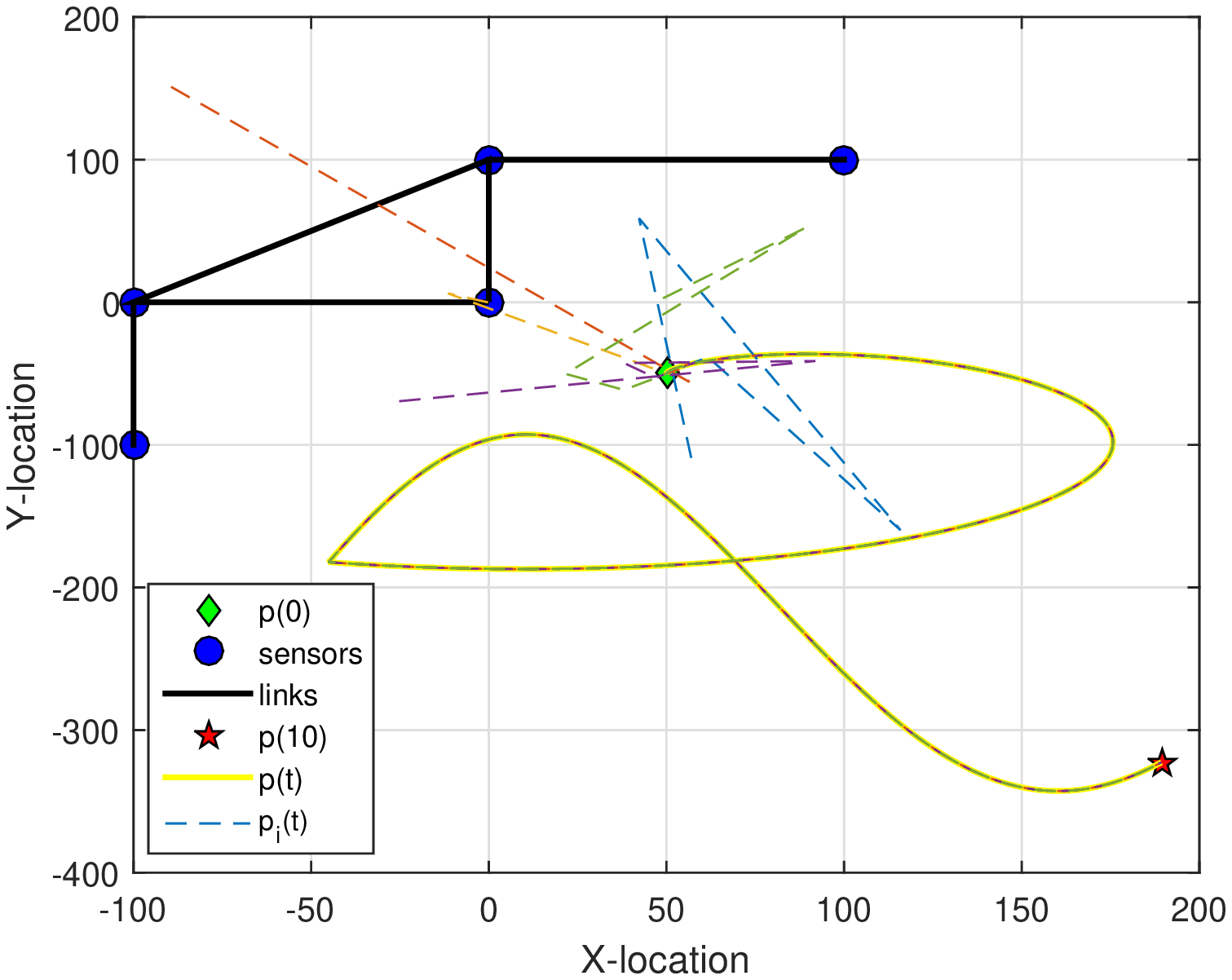}\label{Fig1}}
      \subfigure[Tracking error for all 5 nodes]{
      \psfrag{time}[][]{\footnotesize{$t$}}
      \psfrag{RMSE}[][]{\footnotesize{RMSE}}
      \includegraphics[width=.3\textwidth]{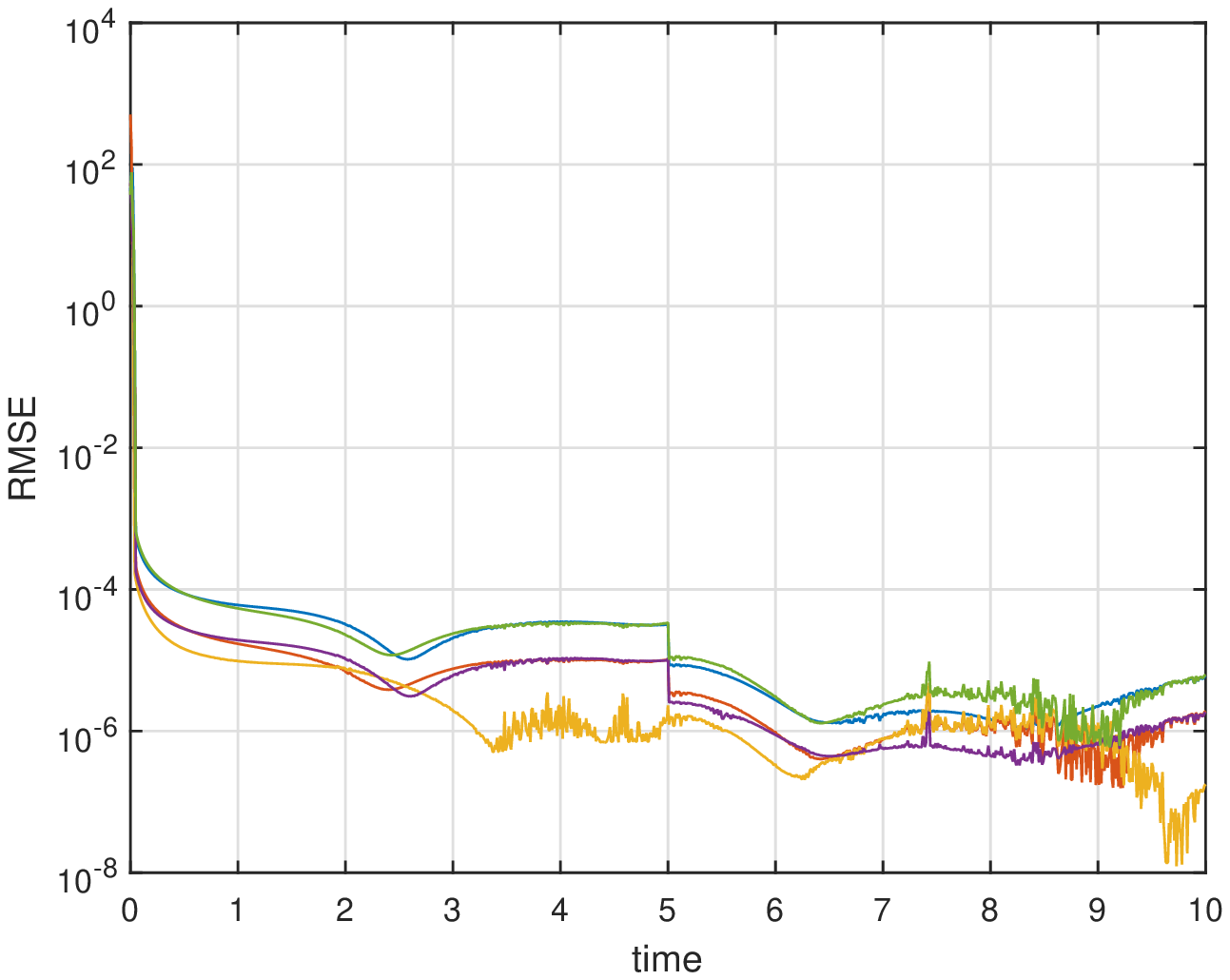}\label{Fig2}}
      \subfigure[Consensus error for individual agents]{
      \psfrag{time}[][]{\footnotesize{$t$}}
      \psfrag{RMSE}[][]{\footnotesize{MSCE}}
      \includegraphics[width=.3\textwidth]{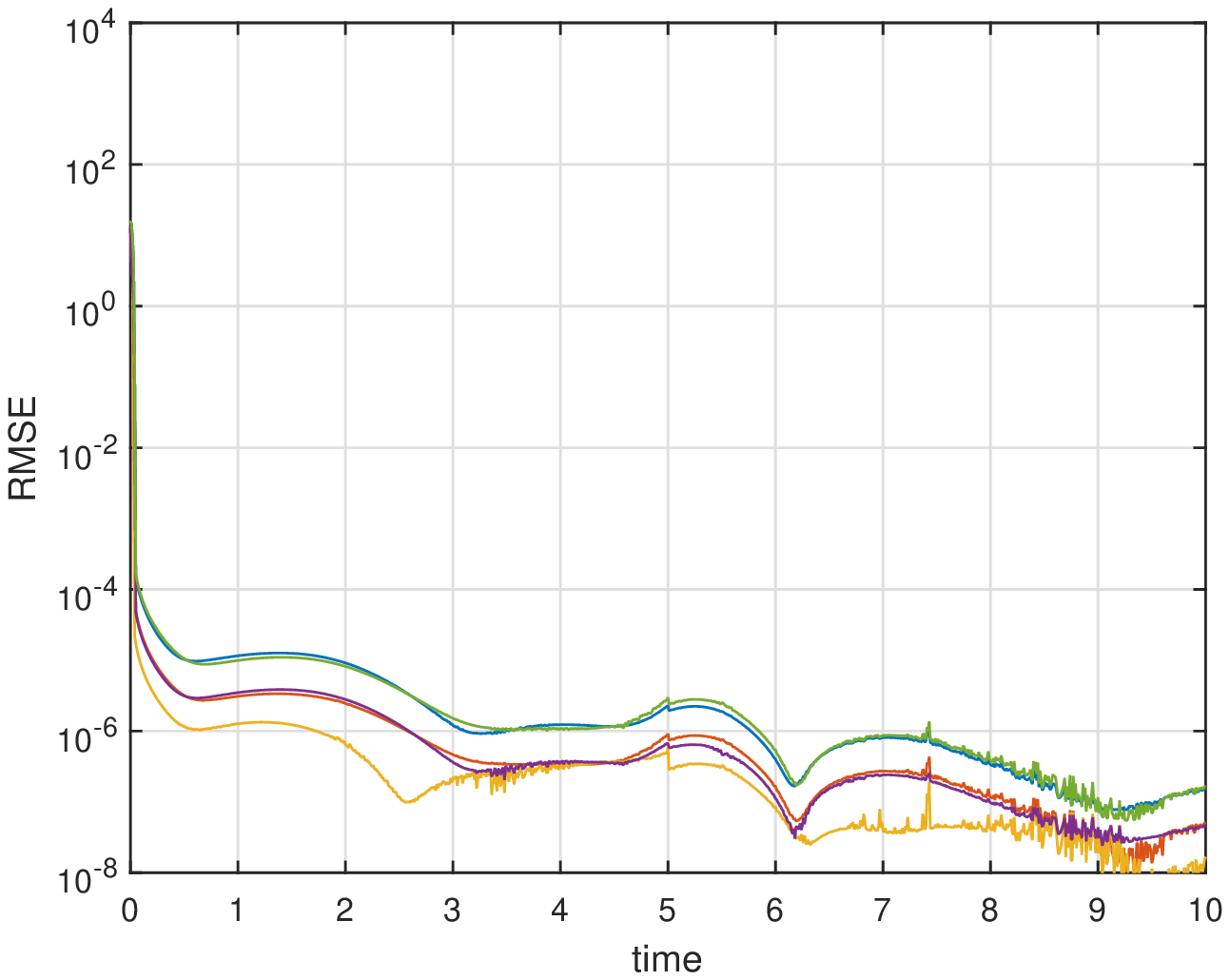}\label{Fig3}}
      \vspace{-0.2cm}
      \caption{Simulation scenario and mean-square errors.}
      \label{Plot2}
  \end{centering}
  \vspace{-0.3cm}
\end{figure*}

\vspace{-3mm}
{\subsection{Convergence Result}}
\vspace{-2mm}
The following theorem illustrates how to select the parameter $\beta$ and the initial conditions $\wb(t_0)$ such that the average-consensus error converges to zero in finite time.

\begin{theorem}\label{Theorem01}
Given Assumptions~\ref{Assump:Phi} and~\ref{Assump:Graph}, the robust dynamic average-consensus algorithm in \eqref{RDATsys2} guarantees that the average-consensus error, $\tilde{\mathbf{x}}(t)$, is globally finite-time convergent, i.e., $\forall\,{\tilde{\mathbf{x}}}(t_0)$, we have $\tilde{\mathbf{x}}(t) = \mathbf{0}$ for all $t\geq t^*$, where
\begin{equation}\label{eq:t*}
t^* = t_0 + \frac{ \| \tilde{\mathbf{x}}(t_0)\|_2}{\lambda_2(L)},
\end{equation}
if $\,\wb(t_0)$ is set to zero and $\beta$ is selected such that
\begin{align}\label{betaEqn}
  \beta \geq 1 + \gamma\, \frac{\sqrt{\hat{n}}}{\hat{\lambda}_2},
\end{align}
where $\hat{n}$ and $\hat{\lambda}_2$ are positive constants such that $\hat{n} \geq n$ and $\hat{\lambda}_2 \leq \lambda_2(L)$, where $\lambda_2(L)$ is the algebraic connectivity of the network.
\end{theorem}

\begin{proof}
Note $\left( \mathcal{B}^{\top} \otimes I_{6}\right) \tilde{\mathbf{x}}(t)$ $=$ $\left( \mathcal{B}^{\top} \otimes I_{6}\right) {\mathbf{x}}(t)$. Thus, after substituting \eqref{RDATsys2a}, the error dynamics can be written as
\begin{align*}
    \dot{\tilde{\mathbf{x}}}(t) = -\beta\, \left( \mathcal{B} \otimes I_{6}\right) \sgn\left\{ \left( \mathcal{B}^{\top} \otimes I_{6}\right)\tilde{\mathbf{x}}(t)\right\} + \left( M \otimes I_{6}\right) \dot{\bm{\phi}}(t).
\end{align*}
From \eqref{AvgConErr}, we have $\tilde{\mathbf{x}}(t_0) = \wb(t_0) + \left( M \otimes I_{6}\right) \bm{\phi}(t_0)$. Since $\wb(t_0) = \mathbf{0}$ and $\mathbf{1}^{\top}_n M = \mathbf{0}^{\top}_n$, we have $\left( \mathbf{1}^{\top}_n \otimes I_{6}\right)\tilde{\mathbf{x}}(t_0) = \mathbf{0}_6$. Since $\mathbf{1}^{\top}_n B = \mathbf{0}^{\top}_{\ell}$, we have $\left( \mathbf{1}^{\top}_n \otimes I_{6}\right)\dot{\tilde{\mathbf{x}}}(t) = \mathbf{0}_6$ and therefore, $\left( \mathbf{1}^{\top}_n \otimes I_{6}\right)\tilde{\mathbf{x}}(t) = \mathbf{0}_6$ for all $t\geq t_0$. Thus we have $\tilde{\mathbf{x}}^{\top}(t) \left(\mathcal{L}\otimes I_{6}\right) \tilde{\mathbf{x}}(t) \geq \lambda_2(\mathcal{L})\| \tilde{\mathbf{x}}(t) \|_2^2$.

Now consider a nonnegative function of the form $V = \frac{1}{2} \tilde{\mathbf{x}}^{\top}(t) \tilde{\mathbf{x}}(t)$. Therefore,
\begin{align*}
\dot{V} &= -\beta\tilde{\mathbf{x}}^{\top}(t) \left( \mathcal{B} \otimes I_{6}\right) \sgn\left\{ \left( \mathcal{B}^{\top} \otimes I_{6}\right)\tilde{\mathbf{x}}(t)\right\} + \\& \qquad \tilde{\mathbf{x}}^{\top}(t) \left(\mathcal{B}\otimes I_{6}\right) \left(\mathcal{B}^{\top}\left(\mathcal{BB}^{\top}\right)^+\otimes I_{6}\right) \dot{\bm{\phi}}(t),
\end{align*}
where we substituted \eqref{Eq:M} for $M$. Thus,
\begin{align*}
\dot{V}
&\leq  -\beta \|  \tilde{\mathbf{x}}^{\top}(t) \left(\mathcal{B}\otimes I_{6}\right) \|_1 +  \|  \tilde{\mathbf{x}}^{\top}(t) \left(\mathcal{B}\otimes I_{6}\right) \|_1 \\
& \quad  \times \| \left(\mathcal{B}^{\top}\otimes I_{6}\right) \|_{\infty}  \| \left(\left(\mathcal{BB}^{\top}\right)^+\otimes I_{6}\right) \|_{\infty} \| \dot{\bm{\phi}}(t) \|_{\infty}.
\end{align*}
Note $ \| \left(\mathcal{B}^{\top}\otimes I_{6}\right) \|_{\infty}  = \| \mathcal{B}^{\top} \|_{\infty} = 2\,$ and $\,\| \left(\mathcal{BB}^{\top}\right)^+ \|_{\infty}$  $\leq \frac{\sqrt{n}}{2\lambda_2\left(\mathcal{L}\right)}$. Thus, if $\beta$ is selected such that \eqref{betaEqn} is satisfied, then we have
\begin{align*}
\dot{V} &\leq - \|  \left(\mathcal{B}^{\top}\otimes I_{6}\right)\tilde{\mathbf{x}}(t) \|_1 \leq - \sqrt{ \|  \left(\mathcal{B}^{\top}\otimes I_{6}\right)\tilde{\mathbf{x}}(t) \|_2^2}\\
&\leq - \sqrt{ \tilde{\mathbf{x}}^{\top}(t) \left(\mathcal{B}\otimes I_{6}\right) \left(\mathcal{B}^{\top}\otimes I_{6}\right)\tilde{\mathbf{x}}(t) } \\
&= -\sqrt{ 2 \tilde{\mathbf{x}}^{\top}(t) \left(\mathcal{L}\otimes I_{6}\right) \tilde{\mathbf{x}}(t) } \leq - \sqrt{2}\sqrt{\lambda_2(\mathcal{L})}\sqrt{V}.
\end{align*}
Thus we have $\frac{1}{2\sqrt{V}}\dot{V} \leq -\frac{1}{2} \sqrt{2\lambda_2(\mathcal{L})}$. Now based on the Comparison Lemma (Lemma 3.4 of \cite{khalil2002nonlinear}), $\sqrt{V(t)} \leq \sqrt{V(t_0)} - \frac{1}{2}\sqrt{2\lambda_2(\mathcal{L})}\, t$.
Since $\dot{V}(t)$ is negative definite and $V(t)$ is positive definite, we have $\tilde{\mathbf{x}}(t) = \mathbf{0}$ for all $t \geq t^*$, where $t^* = t_0 + \frac{\| \tilde{\mathbf{x}}(t_0) \|_2}{\sqrt{\lambda_2(\mathcal{L})}}$.
This concludes the proof.
\end{proof}

\begin{remark}
Note that the robust dynamic average-consensus algorithm in \eqref{RDATsys2} only requires a conservative upper-bound $\hat{n}$ and a lower-bound $\hat{\lambda}_2$. Precise values of $n$ and $\lambda_2(L)$ are not assumed known. There exist several works \cite{GARIN201213,Terelius2012,Charalambous2016} that propose distributed algorithms to estimate the bounds on network size and the algebraic connectivity of the network.
\end{remark}

\vspace{-3mm}
\subsection{Distributed Tracking Algorithm}
\vspace{-2mm}
The dynamic average-consensus algorithm given in \eqref{RDATsys1} guarantees that $\forall i\in\AC$, $\bm{x}_i(t) = \bar{\phib}(t)$ for all $t\geq t^*$, where $t^*$ is given in \eqref{eq:t*}. Let $\Pb_{\bm{x}_i}(t)$ denoted the $2\times 2$ symmetric matrix constructed from the first $4$ entries of $\bm{x}_i(t)$. Also, let $\qb_{\bm{x}_i}(t)$ denoted the $2\times 1$ vector constructed from the last $2$ entries of $\bm{x}_i(t)$. Now each agent computes the least-squares solution as
\begin{equation}\label{SolnOpty}
  \p_i(t) = \left( \Pb_{\bm{x}_i}(t) \right)^{-1} \qb_{\bm{x}_i}(t).
\end{equation}
A summary of the proposed algorithm is given in Algorithm~\ref{Algorithm1}. Now we have the following result:

 \begin{algorithm}
 \caption{Distributed tracking algorithm}
 \label{Algorithm1}
 \begin{algorithmic}[2]
  \STATE \textit{Initialization} :$\quad \wb(t_0) = \mathbf{0}_{6n}$
  \FOR {$t\geq t_0$}
  \FOR {$i = 1$ to $n$}
  \STATE \textit{Obtain}: $z_i(t)$ \& $\hb_i^\top(t)$
  \STATE $\Pb_i(t) = \hb_i(t)\hb^\top_i(t)$
  \STATE $\qb_i(t) = z_i(t)\hb_i(t)$
  \STATE $\phib_i(t) = \begin{bmatrix} \Vect\left( \Pb_i(t) \right) \\ \qb_i(t) \end{bmatrix}$
  \STATE $\bm{x}_i(t) =  \bm{w}_i(t) + \phib_i(t)$
  \STATE $\dot{\bm{w}}_i(t) = - \beta\sum\limits_{j=1}^{n} a_{ij} \sgn \bigg\{ \bm{x}_i(t)-\bm{x}_j(t) \bigg\}$
  \STATE $\Pb_{\bm{x}_i}(t)$ $\Leftarrow$ $\left[ \bm{x}_i(t) \right]_{1:4}$
  \STATE $\qb_{\bm{x}_i}(t)$ $\Leftarrow$ $\left[ \bm{x}_i(t) \right]_{5:6}$
  \STATE $\p_i(t) = \left( \Pb_{\bm{x}_i}(t) \right)^{-1} \qb_{\bm{x}_i}(t)$ \label{Step13}
  \ENDFOR
  \ENDFOR
 \end{algorithmic}
 \end{algorithm}
\vspace{-2mm}
\begin{theorem}\label{Theorem02}
Given Assumptions~\ref{Assump1}, \ref{Assump:Phi}, and \ref{Assump:Graph}, the proposed distributed approach 
guarantees that the individual solutions $\p_i(t)$ converges to the optimal solution $\p^*(t)$ in finite time, i.e., for all $t\geq t^*$,
\begin{equation}\label{Eqn:end}
  \p_i(t) =  \p^*(t),\quad \forall i\in\AC,
\end{equation}
where $t^*$ is given in \eqref{eq:t*}.
\end{theorem}
\begin{proof}
It follows from the finite-time convergence of the dynamic average-consensus algorithm that for all $t\geq t^*$, $\Pb_{\bm{x}_i}(t) = \bar{\Pb}(t)$ and $ \qb_{\bm{x}_i}(t) = \bar{\qb}(t).$ Thus, for all $t\geq t^*$,
\begin{equation*}
  \p_i(t) =  \left( \bar{\Pb}(t) \right)^{-1} \bar{\qb}(t) = \p^*(t),\quad \forall i\in\AC.
\end{equation*}
\end{proof}
\vspace{-2mm}
\begin{remark}
It is important to realize that here it is assumed that a dynamic model for $\p(t)$ is not available to any of the sensors. Obviously, if such information is available, a recursive filter such as a Kalman filter may be employed.
\end{remark}

\section{Numerical Results}\label{sec:sim}
\vspace{-1mm}
Consider the problem of distributed tracking of a maneuvering target using bearing measurements. Figure~\ref{Fig1} depicts the simulation scenario considered, where the sensors are denoted as blue circles, the communication links between the sensors are represented as solid black lines, and the starting and end points of the target trajectory are denoted as a green diamond and red star, respectively. Here $t_0 = 0$ and $t_f = 10$. For the entire duration of simulation, the true target trajectory is given in Fig.~\ref{Fig1} as a thick, solid, yellow line while the individual sensor estimates are given as thin dashed lines. For numerical simulations, we select $\gamma = 10^2$, $\hat{n}=5$, and $\hat{\lambda}_2 = 0.4$.

Notice that the large initial errors in individual estimates are due to the initial error $\tilde{\mathbf{x}}(0)$. Figure~\ref{Fig2} contains the root-mean-square error (RMSE) for the individual sensors for the simulation. Here RMSE of the $i$-th agent is calculated as
$$\text{RMSE}_i(t) = \sqrt{ \frac{1}{2} \left(\p^*(t) - {\p}_i(t)\right)^\top \left(\p^*(t) - {\p}_i(t)\right) }.$$

Figure~\ref{Fig2} indicates that the agents are able to precisely estimate the target trajectory despite the initial error. Note that the non-zero tracking error in the order of $10^{-6}$ is due to the selected integration step size and it can be further decreased by selecting a smaller step size. Figure~\ref{Fig3} contains the mean-square-consensus error (MSCE) for the individual agents calculated as
$$\text{MSCE}_i(t) = \sqrt{ \frac{1}{6} \left(\bar{\phib}(t) - {\bm{x}}_i(t)\right)^\top\left(\bar{\phib}(t) - {\bm{x}}_i(t)\right) }.$$ 

\section{Conclusion}\label{sec:con}

Here we presented a novel distributed algorithm that allows the networked agents to precisely track a highly maneuvering target from bearing measurements. The proposed scheme, built on the dynamic average consensus algorithm, guarantees that the tracking error obtained by individual agents converges to zero in finite time. The proposed continuous-time formulation can be extended to discrete-time scenarios after replacing the discontinuous signum function with an appropriate continuous approximation such as a saturation function. Future research include extending the current approach to highly noisy scenarios and considering privacy preserving event-triggered communication schemes.

\bibliography{Biblio}
\bibliographystyle{IEEEtran}

\end{document}